\documentclass[11pt]{article} 

\usepackage[left=3.0cm, right=3.0cm, top=2.5cm, bottom=3.0cm]{geometry}

\usepackage[T1]{fontenc}
\usepackage[utf8]{inputenc}

\usepackage[numbers,comma,sort]{natbib}
\usepackage[french,english]{babel}

\usepackage{amsmath,amssymb,amsfonts,amsthm,bm}
\usepackage{mathrsfs}
\usepackage{url}

\usepackage[colorlinks=true]{hyperref}
\hypersetup{colorlinks, 
	    citecolor=black,  	
            linkcolor=[rgb]{0.05,0.20,0.55}}

\usepackage{bbm}

\usepackage{sectsty}		
\sectionfont{\normalfont  \bfseries} 

\subsectionfont{\normalfont \bfseries }

\subsubsectionfont{\normalfont \bfseries }

\usepackage{enumerate}
\usepackage{enumitem}
\newlist{hyp0}{enumerate}{1}
\setlist[hyp0]{label=(H\arabic*)}

\usepackage{amsmath}

\numberwithin{equation}{section}

\newtheorem{theorem}{Theorem}
\newtheorem*{theorem*}{Theorem}
\newtheorem{lemma}{Lemma}[section]

\newtheorem{example}[lemma]{Example}
\newtheorem{definition}[lemma]{Definition}
\newtheorem{proposition}[lemma]{Proposition}
\newtheorem{remark}[lemma]{Remark}
\newtheorem{assumption}[lemma]{Assumption}

\usepackage{fancyhdr}
\fancypagestyle{empty}{\fancyfoot[C]{\vspace*{-1.0\baselineskip}\thepage}}

\usepackage[titletoc]{appendix}
\AtBeginDocument{}
\appendixtocon

\begin{document}

\newcommand{\A}{\mathbb{A}}
\newcommand{\C}{\mathbb{C}}
\newcommand{\D}{\mathbb{D}}
\newcommand{\R}{\mathbb{R}}
\newcommand{\ER}{\mathbb{R}}
\newcommand{\Q}{\mathbb{Q}}
\newcommand{\Z}{\mathbb{Z}}
\newcommand{\N}{\mathbb{N}}
\newcommand{\EE}{\mathbb{E}}
\newcommand{\ES}{\mathbb{E}}
\newcommand{\PP}{\mathbb{P}}
\newcommand{\Pcal}{\mathcal{P}}
\newcommand{\Tcal}{\mathcal{T}}
\newcommand{\Rcal}{\mathcal{R}}
\newcommand{\PE}{\mathbb{P}}
\newcommand{\dd}{\textrm{d}}
\newcommand{\usn}{\mathbf{u}_s^n}
\newcommand{\tRoot}{\sigma}

\newcommand{\Tc}{\mathcal{T}}
\newcommand{\Om}{\Omega}
\newcommand{\Omb}{\overline \Omega}
\newcommand{\om}{\omega}
\newcommand{\omb}{\bar \omega}
\newcommand{\Fc}{\mathcal{F}}
\newcommand{\F}{\mathbb{F}}
\newcommand{\T}{\mathbb{T}}
\newcommand{\V}{\mathbb{V}}
\newcommand{\Ac}{\mathcal{A}}
\newcommand{\x}{\times}
\newcommand{\Fbb}{\overline \F}
\newcommand{\Fcb}{\overline \Fc}
\newcommand{\Pc}{\mathcal{P}}
\newcommand{\Pcb}{\overline \Pc}
\newcommand{\Pb}{\overline \PP}
\newcommand{\Uc}{\mathcal{U}}
\newcommand{\Rc}{\mathcal{R}}
\newcommand{\eps}{\varepsilon}

\newcommand{\red}{\textcolor{red}}
\newcommand{\blue}{\textcolor{blue}}
\title{ \Large{\textbf{On the Root solution to the Skorokhod embedding problem \\ given full marginals }}}

\author{Alexandre Richard \footnote{Université Paris-Saclay, CentraleSupélec, MICS and CNRS FR-3487, alexandre.richard@centralesupelec.fr. This author is grateful to the CMAP for its hospitality, where most of this work was carried out.}
	\and Xiaolu Tan \footnote{Department of Mathematics, The Chinese University of Hong Kong. This author acknowledges the financial support of the Initiative de Recherche ``M\'ethodes non-lin\'eaires pour la gestion des risques financiers'' sponsored by AXA Research Fund,  xiaolu.tan@cuhk.edu.hk.
	This author acknowledges the financial support of the Initiative de Recherche ``M\'ethodes non-lin\'eaires pour la gestion des risques financiers'' sponsored by AXA Research Fund.}
	\and Nizar Touzi \footnote{CMAP, Ecole Polytechnique, nizar.touzi@polytechnique.edu. This author acknowledges the financial support of the Chaires Financial Risks, and Finance and Sustainable Development, hosted by the Louis Bachelier Institute.}
}

\date{\today}

\renewcommand{\thefootnote}{\fnsymbol{footnote}} 
\footnotetext{\emph{MSC2010 Subject Classification.} 60G40, 91G80.}     
\footnotetext{This work was started while the authors had the opportunity to benefit from the ERC grant 311111 RoFiRM 2012-2018.}    
\renewcommand{\thefootnote}{\arabic{footnote}} 

\maketitle

\begin{abstract}
 This paper examines the Root solution of the Skorokhod embedding problem given full marginals on some compact time interval. Our results are obtained by limiting arguments based on finitely-many marginals Root solution of Cox, Ob{\l}{\'o}j, and Touzi \cite{COT}. Our main result provides a characterization of the corresponding potential function by means of a convenient parabolic PDE.
\end{abstract}

\section{Introduction}

	The Skorokhod embedding problem, initially suggested by Skorokhod \cite{Skorokhod}, 
	consists in finding a stopping time $\tau$ on a given Brownian motion $B$ such that $B_{\tau} \sim \mu$ for a given marginal distribution $\mu$ on $\R$. The existing literature contains various solutions suggested in different contexts. Some of them satisfy an optimality property among all possible solutions, e.g. the Root solution \cite{Root}, the Rost solution \cite{Rost}, the Az\'ema-Yor solution \cite{AzemaYor}, the Vallois solution \cite{Vallois}, the Perkins solution \cite{Perkins}, etc.
	This problem has been extensively revived in the recent literature due to the important connexion with the problem of robust hedging in financial mathematics. We refer to Ob{\l}{\'o}j \cite{Obloj} and Hobson \cite{Hobson} for a survey on different solutions and the applications in finance.

	\vspace{0.5em}

	Our interest in this paper is on the Root solution of the Skorokhod embedding problem, which is characterized as a hitting time of the Brownian motion $B$ of  some time-space domain $\Rc$ unlimited to the right, that is, $\tau_{R} := \inf\{ t \ge 0~: (t, B_t) \in \Rc \}$.
	This solution was shown by Rost \cite{Rost} to have the minimal variance among all solutions to the embedding problem.
	As an application in finance, it can be used to deduce robust no-arbitrage price bounds for a class of variance options (see e.g. Hobson \cite{Hobson}).
	To find the barrier $\Rc$ in the description of the Root solution, Cox and Wang \cite{CoxWang} provided a construction by solving a variational inequality.
	This approach is then explored in Gassiat, Oberhauser and dos Reis \cite{GOR} and Gassiat, Mijatovi{\'c} and Oberhauser \cite{GMO} to construct $\Rc$ under more general conditions. We also refer to the remarkable work of Beiglb{\"o}ck, Cox and Huesmann \cite{BCH} which derives the Root embedding, among other solutions, as a natural consequence of the monotonicity principle in optimal transport.

	\vspace{0.5em}

	It is also natural to extend the Skorokhod embedding problem to the multiple-marginals case. Let $(\mu_k)_{0 \le k \le n}$ be a family of marginal distributions, nondecreasing in the convex order, i.e. $\mu_{k-1}(\phi)\le \mu_k(\phi)$, $k=1,\ldots,n$, for all convex functions $\phi: \R \to \R$. The multiple-marginals Skorokhod embedding problem is to find  a Brownian motion $B$, together with an increasing sequence of stopping times $(\tau_k)_{1 \le k \le n}$, such that 
	$B_{\tau_k} \sim \mu_k$ for each $k=1, \cdots, n$.
	Madan and Yor \cite{MadanYor} provided a sufficient condition on the marginals, under which the Az\'ema-Yor embedding stopping times corresponding to each marginal are automatically ordered,
	so that the iteration of Az\'ema-Yor solutions provides a solution to the multiple-marginals Skorokhod embedding problem.
	In general, the Az\'ema-Yor embedding stopping times may not be ordered. An extension of the Az\'ema-Yor embedding was obtained by	Brown, Hobson and Rogers \cite{BHR} in the two-marginals case, and later by Ob{\l}{\'o}j and Spoida \cite{OS} for an arbitrary finite number of marginals. Moreover, the corresponding embeddings enjoys the similar optimality property as in the one-marginal case.
	In Claisse, Guo and Henry-Labord{\`e}re \cite{CGHL}, an extension of the Vallois solution to the two-marginals case is obtained for a specific class of marginals. We also refer to Beiglb{\"o}ck, Cox and Huesmann \cite{BCH_n} for a geometric representation of the optimal Skorokhod embedding solutions given multiple marginals.
	
	The Root solution of the Skorokhod embedding problem was recently extended by Cox, Ob{\l}{\'o}j and Touzi \cite{COT} to the multiple-marginals case. Our objective in this paper is to characterize the limit case with a family of full marginals $\mu = (\mu_t)_{t \in [0,1]}$.
	Let us assume that each $\mu_t$ has finite first moment and $t \mapsto \mu_t$ is increasing in convex order. 
	Such a family was called a peacock (or PCOC ``Processus Croissant pour l'Ordre Convexe'' in French) by Hirsch, Profeta, Roynette and Yor \cite{HPRY}.
	Then Kellerer's Theorem \cite{Kellerer} ensures the existence of a martingale $M = (M_t)_{0 \le t \le 1}$ such that $M_t \sim \mu_t$ for each $t \in [0,1]$. If in addition $t \mapsto \mu_t$ is right-continuous, then so is the martingale $M$ up to a modification.
	Further, by Monroe's result \cite{Monroe}, one can find an increasing sequence of stopping times $(\tau_t)_{0 \le t \le 1}$ together with a Brownian motion $B = (B_s)_{s \ge 0}$ such that $B_{\tau_t} \sim \mu_t$ for each $t \in [0,1]$.
	This consists in an embedding for the full marginals $\mu$.
	We refer to \cite{HPRY} for different explicit constructions of the martingales or embeddings fitting the peacock marginals.
	Among all martingales or $\mu$-embeddings, it is interesting to find solutions enjoying some optimality properties.
	In the context of Madan and Yor \cite{MadanYor}, the Az\'ema-Yor embedding $\tau^{AY}_t$ of the one-marginal problem with $\mu_t$ is ordered w.r.t. $t$, and thus $(\tau^{AY}_t)_{0 \le t \le 1}$ is the embedding maximizing the expected maximum among all embedding solutions. This optimality is further extended by K{\"a}llblad, Tan and Touzi\cite{KTT} allowing for non-ordered barriers.
	Hobson \cite{HobsonPeacock} gave a construction of a martingale with minimal expected total variation among all martingales fitting the marginals.
	Henry-Labord{\`e}re, Tan and Touzi\cite{HLTT_Brenier} provided a local L\'evy martingale, as limit of the left-monotone martingales introduced by Beiglb{\"o}ck and Juillet \cite{BJ} (see also Henry-Labord{\`e}re and Touzi \cite{HLT_Brenier}),
	which inherits its optimality property.
	For general existence of the optimal solution and the associated duality result, one needs a tightness argument,
	which is studied in Guo, Tan and Touzi \cite{GTT_S_topology} by using the S-topology on the Skorokhod space,
	and in K{\"a}llblad, Tan and Touzi \cite{KTT} by using the Skorokhod embedding approach.

	\vspace{0.5em}
	
	The aim of this paper is to study the full marginals limit of the multiple-marginals Root embedding as derived in \cite{COT}.
	This leads to a natural extension of the Root solution for the embedding problem given full marginals.
	Using the tightness result in \cite{KTT}, we can easily obtain the existence of such limit as well as its optimality.
	We then provide some characterization of the limit Root solution as well as that of the associated optimal stopping problem, which is used in the finitely many marginals case to describe the barriers.

	\vspace{0.5em}
	
	In the rest of the paper, 
	we will first formulate our main results in Section \ref{sec:MainResults}.
	Then in Section \ref{sec:recall}, we recall some details on the Root solution given finitely many marginals in \cite{COT} and the limit argument of \cite{KTT}, which induces the existence of the limit Root solution for the embedding problem given full marginals.
	We then provide the proofs of our main results on some characterization of the limit Root solution in Section \ref{sec:proofs}.
	Some further discussions are finally provided in Section \ref{sec:discussion}.

\section{Problem formulation and main results}
\label{sec:MainResults}
	We are given a family of probability measures 
	$\mu = (\mu_s)_{s\in[0,1]}$ on $\R$, such that $\mu_s$ is centred with finite first moment for all $s \in [0,1]$,
	$s \mapsto \mu_s$ is c\`adl\`ag under the weak convergence topology, and
	the family $\mu$ is non-decreasing in convex order, i.e. for any convex function $\phi:\R \to \R$,
	\begin{align*}
		\int_\R \phi(x) \mu_s(dx) \leq \int_\R \phi(x) \mu_t(dx)~
		~~\mbox{for all}~s \le t.
	\end{align*}

\begin{definition}\label{def:embedding}
	$\mathrm{(i)}$ A \emph{stopping rule} is a term 
	$$
		\alpha = (\Om^{\alpha}, \Fc^{\alpha}, \F^{\alpha}, \PP^{\alpha}, B^{\alpha},\mu_0^\alpha, (\tau^{\alpha}_s)_{s\in [0,1]}),
	$$
	such that $(\Om^{\alpha}, \Fc^{\alpha}, \F^{\alpha}, \PP^{\alpha})$ is a filtered probability space equipped with a Brownian motion $B^{\alpha}$ with initial law $\mu_0^\alpha$
	and a family of stopping times $(\tau^{\alpha}_s)_{s \in [0,1]}$ such that $s \mapsto \tau^{\alpha}_s$ is c\`adl\`ag, non-decreasing and $\tau_0^\alpha=0$.
	We denote
	$$
		\Ac := \big\{ \mbox{All stopping rules}\big\},
		~~\mbox{and}~~
		\Ac_t := \big\{ \alpha \in \Ac ~:  \tau^{\alpha}_1 \le t \big\},
		~~\mbox{for all}~t \ge 0.
	$$
	Denote also
	$$
		\Ac^0 ~:=~ \{ \alpha \in \Ac ~: \mu_0^\alpha = \delta_0 \}
		~~~\mbox{and}~~
		 \Ac^0_t ~:=~  \{ \alpha \in \Ac ~: \tau^{\alpha}_1 \le t,~~\mu_0^\alpha = \delta_0 \}.
	$$

	\noindent $\mathrm{(ii)}$ A stopping rule $\alpha \in \Ac$  is called a \emph{$\mu$-embedding} if
	 $\mu^{\alpha}_0 = \mu_0$,
	$(B^{\alpha}_{t \wedge \tau^{\alpha}_1})_{t \ge 0}$ is uniformly integrable
	and $B^{\alpha}_{\tau^{\alpha}_s} \sim \mu_{s}$ for all $s \in [0,1]$. 
	In particular, one has $\tau^{\alpha}_0 = 0$, a.s. 
	We denote by $\Ac(\mu)$ the collection of all $\mu$-embeddings. 

	\vspace{0.5em}

	\noindent $\mathrm{(iii)}$ Let $\pi_n= \{0 = s_0 < s_1 < \cdots < s_n =1\}$ be a partition of $[0,1]$. 
	A stopping rule $\alpha \in \Ac$  is called a \emph{$(\mu, \pi_n)$-embedding} if
	 $\mu^{\alpha}_0 = \mu_0$,
	$(B^{\alpha}_{t \wedge \tau^{\alpha}_1})_{t \ge 0}$ is uniformly integrable 
	and $B^{\alpha}_{\tau^{\alpha}_{s_k}} \sim \mu_{s_k}$ for all $k = 0, \dots, n$.
	We denote by $\Ac(\mu, \pi_n)$ the collection of all $(\mu, \pi_n)$-embeddings.

\end{definition}

	Our aim is to study the Root solution of the Skorokhod embedding problem (SEP, hereafter) given full marginals $(\mu_s)_{s\in[0,1]}$.
	To this end, we first recall the Root solution of the SEP given finitely many marginals, constructed in \cite{COT}.
	Let $(\pi_n)_{n \ge 1}$ be a sequence of partitions of $[0,1]$, where  $\pi_n = \{0 = s^n_0 < s^n_1 < \cdots < s^n_n =1\}$ and $|\pi_n| := \max_{k=1}^n |s^n_k - s^n_{k-1}| \to 0$ as $ n \to \infty$.
	Then for every fixed $n$, one obtains $n$ marginal distributions $(\mu_{s^n_k})_{1 \le k \le n}$ and has the following Root solution to the corresponding SEP.
	
	\begin{theorem*}[Cox, Ob\l{\'o}j and Touzi, 2018]
		For any $n \ge 1$, there exists a $(\mu, \pi_n)$-embedding $\alpha^*_n$ called Root embedding,
		where $\sigma^n_k := \tau^{\alpha^*_n}_{s^n_k}$ is defined by
		\begin{eqnarray*}
		\sigma^n_0 := 0
		&\mbox{ and }&
		\sigma^n_k := \inf \{ t \ge \sigma^n_{k-1} ~: (t, B^{\alpha^*_n}_t) \in \Rc^n_k \},
		\end{eqnarray*}
		for some family of barriers $(\Rc^n_k)_{1 \le k \le n}$ in $\R_+ \x \R$.
		Moreover, for any non-decreasing and non-negative function $f: \R_+ \to \R_+$, one has
		$$
			\EE^{\PP^{\alpha^*_n}} \Big[ \int_0^{\tau^{\alpha^*_n}_1} f(t) dt \Big]
			~=~
			\inf_{\alpha \in \Ac(\mu, \pi_n)} \EE^{\PP^\alpha} \Big[  \int_0^{\tau^\alpha_1} f(t) dt \Big].
		$$
	\end{theorem*}
	
	The barriers $(\Rc^n_k)_{1 \le k \le n}$ are given explicitly in \cite{COT} by solving an optimal stopping problem, see Section \ref{subsec:Root_n} below.

	\vspace{0.5em}

	Let us denote by $\A([0,1], \R_+)$ the space of all c\`adl\`ag non-decreasing functions $a: [0,1] \to \R_+$,
	which is a Polish space under the L\'evy metric.
	Notice also that the L\'evy metric metrizes the weak convergence topology on $\A([0,1], \R_+)$ seen as a space of finite measures.
	 See also (1.2) in \cite{KTT} for a precise definition of the L\'evy metric on $\A([0,1], \R_+)$.
	Denote also by $C(\R_+, \R)$ the space of all continuous paths $\om: \R_+ \to \R$ with $\om_0 =0$,
	which is a Polish space under the compact convergence topology.
	Then for a given embedding $\alpha$, one can see $(B^{\alpha}_{\cdot},\tau^{\alpha}_{\cdot})$ as a random element taking values in $ C(\R_+, \R) \x \A([0,1], \R_+)$, which allows to define their weak convergence.
	Our first main result ensures that the $(\mu, \pi_n)$-Root embedding has a limit in the sense of the weak convergence,
	which enjoys the same optimality property, and thus can be considered as the full marginals Root solution of the SEP. Our proof requires the following technical condition.
	
	\vspace{0.5em}

	Let  us denote by $U: [0,1] \x \R \to \R$ the potential function of $\mu$:
		\begin{equation} \label{eq:potiential_function}
			U(s,x) := -\int_\R |x-y| \mu_s(dy).
		\end{equation}
	\begin{assumption} \label{assum:U}
		The partial derivative $\partial_s U$ exists and is continuous,
		and $x \mapsto \sup_{s\in[0,1]}\partial_s U(s,\cdot)$ has at most polynomial growth.
	\end{assumption}

	\begin{remark}
		When $\mu_s$ has a density function $y \mapsto f(s, y)$ for every $s \in [0,1]$, 
		and the derivative $\partial_sf(s,y)$ exists and satisfies 
		$$
			C := \int_{\R} (|y| \vee 1) \sup_{0 \le s \le 1} |\partial_s f(s,y)| dy < \infty.
		$$
		Then it is easy to deduce that $\partial_s U(s,x)$ exists and satisfies $|\partial_s U(s,x)| \le C(1 + |x|)$.
		We also refer to Section \ref{subsec:Ass_U} for further discussion and examples on Assumption \ref{eq:potiential_function}.
	\end{remark}

	\begin{theorem} \label{thm:KTTintro}
		$\mathrm{(i)}$ Let $(\pi_n)_{n \ge 1}$ be a sequence of partitions of $[0,1]$ such that $|\pi_n| \to 0$ as $n \to \infty$.
		Denote by $\alpha^*_n$ the corresponding $(\mu, \pi_n)$-Root embedding solution.
		Then there exists $\alpha^* \in \Ac(\mu)$ such that the sequence $(B^{\alpha^*_n}_{\cdot},\tau^{\alpha^*_n}_{\cdot})_{n \ge 1}$ converges weakly to $(B^{\alpha^*}_{\cdot},\tau^{\alpha^*}_{\cdot})$.
		Moreover, for all non-decreasing and non-negative functions $f: \R_+ \to \R_+$, one has
		$$
			\EE^{\PP^{\alpha^*}} \Big[ \int_0^{\tau^{\alpha^*}_1} f(t) dt \Big]
			~=~
			\inf_{\alpha \in \Ac(\mu)} \EE^{\PP^\alpha} \Big[  \int_0^{\tau^\alpha_1} f(t) dt \Big].
		$$
		
		\noindent $\mathrm{(ii)}$ Under Assumption \ref{assum:U}, for all fixed $(s,t)\in[0,1]\times\R_+$, the law of $B^{\alpha^*}_{\tau^{\alpha^*}_s\wedge t}$ is independent of the sequence of partitions $(\pi_n)_{n \ge 1}$ and of the limit $\alpha^*$.
	\end{theorem}

	We next provide some characterization of the full marginals Root solution of the SEP $\alpha^*$ given in Theorem \ref{thm:KTTintro}. Let 
	\begin{eqnarray} \label{eq:def_u}
		u(s,t,x) := -\EE^{\PP^{\alpha^*}} \big[ |B^{\alpha^*}_{t\wedge \tau^{\alpha^*}_s} - x| \big], 
		& \quad (s,t,x)\in\mathbf{Z}:=[0,1] \times\R_+\times \R.&
	\end{eqnarray}
	Our next main result, Theorem \ref{th:OSP} below, provides a unique characterization of $u$ which is independent of the nature of the limit $\alpha^*$, thus justifying Claim (ii) of Theorem \ref{thm:KTTintro}. Moreover, it follows by direct computation that one has
	\begin{equation} \label{eq:linear_growth}
		 -|x|-\sqrt{t}\,\EE |\mathbf{N}(0,1)|
		 \;\le\; U_{\mathbf{N}(0,t)}(x)
		 \;\le\; u(1,t,x)
		 \le\; u(s,t,x)
		 \le\; U(0,x),
	\end{equation}
	for all $(s,t,x)\in [0,1] \x \R_+ \x \R$,
	where we denoted by $U_{\mathbf{N}(0,t)}$ the potential function of the $\mathbf{N}(0,t)$ distribution
	(see \eqref{eq:potiential_function} for the definition of the potential function).
	
	\vspace{0.5em}
	
	In the finitely many marginals case in \cite{COT}, the function $u$ is obtained from an optimal stopping problem and is then used to define the barriers in the construction of the Root solution.
	Similar to equations (2.10) and (3.1) in \cite{COT}, we can characterize $u$ as value function of an optimal stopping problem,
	and then as unique viscosity solution of the variational inequality:
	\begin{equation}\label{eq:PDE}
	\begin{cases}
		\min 
		\big\{
			\partial_t u - \frac{1}{2} \partial^2_{xx} u, 
			~\partial_s(u-U) 
		\big\}
		~=~ 0, &\mbox{on}~~\mbox{\rm int}^p(\mathbf{Z}),\\
		u\big|_{t=0} = u\big|_{s=0} = U(0,.), &
	\end{cases}
	\end{equation}
	where we introduce the parabolic boundary and interior of $\mathbf{Z}$:
	$$
		\partial^p \mathbf{Z} := \{(s,t,x) \in \mathbf{Z} ~:  s \wedge t = 0\}
		~~~\mbox{and}~~
		\mbox{\rm int}^p(\mathbf{Z}) := \mathbf{Z} \setminus \partial^p \mathbf{Z}.
	$$
	 In \eqref{eq:PDE}, the boundary condition $u \big|_{t=0}=U(0,\cdot)$ means that $u(s,0,x)=U(0,x)$ for all $(s,x) \in [0,1] \x \R$. 
	Let $Du$ and $D^2u$ denote the gradient and Hessian of $u$ w.r.t. $z = (s,t,x)$, and set:
	\begin{equation} \label{eq:def_F}
		F(Du, D^2u) := \min 
		\big\{
			\partial_t u - \frac{1}{2} \partial^2_{xx} u, 
			~\partial_s(u-U)
		\big\}.
	\end{equation}

	\begin{definition}
		$\mathrm{(i)}$ An upper semicontinuous function $v: \mathbf{Z} \to \R$ is a viscosity subsolution of \eqref{eq:PDE} if 
		$v|_{\partial^p \mathbf{Z}} \le U(0, \cdot)$ and $F(D\varphi, D^2 \varphi)(z_0) \le 0$ for all $(z_0, \varphi) \in \mbox{\rm int}^p(\mathbf{Z}) \x C^2(\mathbf{Z})$ satisfying $(v-\varphi)(z_0) = \max_{z \in \mathbf{Z}} (v - \varphi)(z)$.

		\vspace{0.5em}

		\noindent $\mathrm{(ii)}$ A  lower semicontinuous function $w: [0,1] \x \R_+ \x \R \to \R$ is a viscosity supersolution of \eqref{eq:PDE} if 
		$w|_{\partial^p \mathbf{Z}} \ge U(0, \cdot)$ and $F(D\varphi, D^2 \varphi)(z_0) \ge 0$ for all $(z_0, \varphi) \in \mbox{\rm int}^p(\mathbf{Z}) \x C^2(\mathbf{Z})$ satisfying $(w-\varphi)(z_0) = \min_{z \in \mathbf{Z}} (w - \varphi)(z)$.

		\vspace{0.5em}

		\noindent $\mathrm{(iii)}$ A continuous function $v$ is a viscosity solution of \eqref{eq:PDE} if it is both viscosity subsolution and supersolution.
	\end{definition}

\begin{theorem}\label{th:OSP}
	Let Assumption \ref{assum:U} hold true.
	
	\noindent $\mathrm{(i)}$ The function $u$ can be expressed as value function of an optimal stopping problem,
	\begin{equation}\label{eq:def_u0}
		u(s,t,x) 
		= 
		\sup_{\alpha \in \Ac_t^0} 
		\EE^{\PP^{\alpha}} \Big[ 
			U(0,x+B^{\alpha}_{\tau^{\alpha}_s}) + \int_0^s \partial_s U(s-k,x+B^{\alpha}_{\tau^{\alpha}_k}) \mathbf{1}_{\{\tau^{\alpha}_k<t\}} dk 
		\Big].
	\end{equation}

	\noindent $\mathrm{(ii)}$ The function $u(s,t,x)$ is non-increasing and Lipschitz in $s$ with a locally bounded Lipschitz constant $C(t,x)$, uniformly Lipschitz in $x$ and uniformly $\frac12$-H\"older in $t$.
	
	\vspace{0.5em}
	
	\noindent $\mathrm{(iii)}$  $u$ is the unique viscosity solution of \eqref{eq:PDE} in the class of functions satisfying
	\begin{eqnarray*}
		|u(s,t,x)| \le C(1+t+|x|), ~~(s,t,x)\in\mathbf{Z},
		&~\mbox{for some constant}~C>0.
	\end{eqnarray*}
\end{theorem}


\section{Multiple-marginals Root solution of the SEP and its limit}
\label{sec:recall}
	The main objective of this section is to recall the construction of the Root solution to the SEP given multiple marginals from \cite{COT}.
	As an extension to the one-marginal Root solution studied in \cite{CoxWang} and \cite{GMO}, the solution to the multiple-marginals case enjoys some optimality property among all embeddings.
	We then also recall the limit argument in \cite{KTT} to show how the optimality property is preserved in the limit.

\subsection{The Root solution of the SEP given multiple marginals}
\label{subsec:Root_n}

	Let $n \in \N$ and $\pi_n$ be a partition of $[0,1]$, with $\pi_n = \{0=s^n_0<s^n_1<\dots<s^n_n=1\}$, 
	we then obtain $n$ marginal distributions $\mu^n := \{\mu_{s^n_j}\}_{j = 1, \cdots, n}$
	and recall the Root solution to the corresponding embedding problem.
	
	Let  $\Om = C(\R_+, \R)$ denote the canonical space of all continuous paths $\om: \R_+ \longrightarrow \R$ with $\om_0 = 0$, 
	$B$ be the canonical process, $B^x:=x+B$, $\F = (\Fc_t)_{t \ge 0}$ be the canonical filtration, $\Fc := \Fc_{\infty}$,
	and $\PP_0$ the Wiener measure under which $B$ is a standard Brownian motion.
	For each $t \ge 0$, let $\Tc_{0,t}$ denote the collection of all $\F$-stopping times taking values in $[0, t]$. Denote
	$$
	\delta^n U(s^n_j,x) \;:=\; U(s_j^n,x) - U(s_{j-1}^n,x),~~x\in\R,
	$$
	which is non-positive since $\{\mu_s\}_{s\in[0,1]}$ is non-decreasing in convex ordering.
	We then define the function $u^n(\cdot)$ by a sequence of optimal stopping problems:
	\begin{equation}\label{eq:OS_discrTime}
	\begin{split}
		&u^n\big|_{s=0} := U(s_0^n, .), ~\mbox{and}~ \\
		&u^n(s_j^n,t,x) 
		:= \sup_{\theta \in \Tc_{0,t}} \EE \left[u^n(s^n_{j-1}, t-\theta, B^x_\theta) + \delta^n U(s^n_j, B^x_{\theta}) \mathbf{1}_{\{\theta<t\}}\right].
	\end{split}
	\end{equation}
	Similarly to \eqref{eq:PDE}, the boundary condition $u^n\big|_{s=0} := U(s_0^n, .)$ means that
	$u^n(0,t,x) = U(s_0^n, x)$ for all $(t,x) \in \R_+ \x \R$.
Denoting similarly $\delta^n u(s^n_j,t,x) = u^n(s_j^n,t,x) - u^n(s_{j-1}^n,t,x)$, we define the corresponding stopping regions
	\begin{equation}\label{eq:defRegion}
		\mathcal{R}_j^n 
		~:=~
		\big\{(t,x)\in [0,\infty]\times [-\infty,\infty]:~ \delta^n u(s^n_j,t,x) = \delta^nU(s^n_j,x) \big\},
		~~j=1,\ldots,n.
	\end{equation}

	Given the above, the Root solution for the Brownian motion $B$ on the space $(\Om, \Fc, \PP_0)$,
	is given by the family $\tRoot^n = (\tRoot_1^n,\dots,\tRoot_n^n)$ of stopping times
	\begin{equation} \label{eq:nRoot}
		\tRoot^n_0 :=0,
		~~\mbox{and}~~
		~\tRoot_j^n :=  \inf \big\{t\geq \tRoot_{j-1}^n:~ (t,B_t)\in \mathcal{R}_j^n \big\},
		~~ \forall j\in \{1,\dots,n\}.
	\end{equation}
	The stopping times $\tRoot^n$ induce a stopping rule $\alpha^*_n$ in the sense of Definition \ref{def:embedding}:
	\begin{equation} \label{eq:nRoot_conti}
		\alpha^*_n 
		~=~
		\big(\Om^{\alpha^*_n}, \Fc^{\alpha^*_n}, \F^{\alpha^*_n}, \PP^{\alpha^*_n}, B^{\alpha^*_n}, \mu_0^{\alpha^*_n}, \tau^{\alpha^*_n} \big) 
		~:=~
		\big(\Om, \Fc, \F, \PP_0, B , \mu_0, \tau^{\alpha^*_n} \big),
	\end{equation}
	with $\tau^{\alpha^*_n}_s := \tRoot^n_j$ for $s \in [s^n_j, s^n_{j+1})$.
	Recall also $\Ac(\mu, \pi_n)$ defined in Definition \ref{def:embedding}.

	\begin{theorem*}[Cox, Ob{\l}{\'o}j and Touzi \cite{COT}]
		The stopping rule $\alpha^*_n$ is a $(\mu, \pi_n)$-embedding, with
		\begin{equation} \label{eq:RepresUn}
			u^n(s^n_j, t, x) ~=~ - \mathbb{E}^{\mu_0} |B_{t \wedge \tRoot^n_j} - x|.
		\end{equation}
Moreover, for all non-decreasing and non-negative $f: \R_+ \to \R_+$, we have		
		\begin{equation*}
			\EE^{\mu_0}\Big[\int_0^{\tRoot^n_n} f(t) dt \Big] 
			~=~
			\EE^{\PP^{\alpha^*_n}}\Big[\int_0^{\tau^{\alpha^*_n}_1} f(t) dt \Big] 
			~=~ 
			\inf_{\alpha \in \Ac(\mu, \pi_n)} \EE^{\PP^{\alpha}}\Big[\int_0^{\tau^{\alpha}_1} f(t) dt \Big].
		\end{equation*}
			\end{theorem*}

	Using a dynamic programming argument, one
	can also reformulate the definition of $u^n$ in \eqref{eq:OS_discrTime} by induction as a global multiple optimal stopping problem.
	Let us denote by $\Tc^n_{0,t}$ the collection of all terms $(\tau_1, \cdots, \tau_n)$, 
	where each $\tau_j$, $j=1, \cdots, n$, is a $\F$-stopping time on $(\Om, \Fc, \PP_0)$ satisfying $0 \le \tau_1 \le \ldots \le \tau_n \le t$.

	\begin{proposition}\label{prop:iterated}
		For all $j = 1, \cdots, n$, we have
		\begin{equation}\label{eq:OS_discrTime_iterated}
			u^n(s^n_j,t,x) 
			= 
			\sup_{(\tau_1, \dots, \tau_n) \in \Tc^n_{0,t}} 
			\EE\Big[U(0,x+B_{\tau_j}) +\sum_{k=1}^j \delta^nU(s^n_k,x+B_{\tau_{j-k+1}}) \mathbf{1}_{\{\tau_{j-k+1}<t\}}\Big].
		\end{equation}
	\end{proposition}
	\begin{proof}
	We will use a backward induction argument.
	First, let us denote by $\Tc_{r,t}$ the collection of all $\F$-stopping times taking values in $[r,t]$, and by $B^{r,x}_s := x + B_s - B_r$ for all $s \ge r$.
	Then it follows from the expression \eqref{eq:OS_discrTime} that
	$$
		u^n (s^n_{j-1}, t - r, x) 
		= 
		\sup_{ \tau \in \Tc_{r,t}} \EE \big[
			u^n(s^n_{j-2}, t - \tau, B^{r,x}_{\tau}) + \delta^n U(s^n_{j-1}, B^{r,x}_{\tau}) \mathbf{1}_{\{\tau < t\}}
		\big].
	$$
	Using the dynamic programming principle, one has
\begin{eqnarray*}
	u^n(s^n_j,t,x) 
	&=&
	\sup_{\tau_1 \in \Tc_{0,t} } 
	\EE\bigg[
		u^n(s^n_{j-1}, t-\tau_1, B^{0,x}_{\tau_1})
		+ \delta^nU(s^n_j,B^{0,x}_{\tau_1}) \mathbf{1}_{\{\tau_1<t\}}
	\bigg] \\
	&= & 
	\sup_{\tau_1 \in \Tc_{0,t} } 
	\EE\bigg[
		\mathrm{ess}\sup_{ \tau_2 \in \Tc_{\tau_1,t}} \EE \Big[
			u^n(s^n_{j-2}, t- \tau_2, B^{\tau_1,B^{0,x}_{\tau_1}}_{\tau_2}) \\
	&&~~~~~~~~~~~~~
		+ \delta^n U(s^n_{j-1}, B^{\tau_1, B^{0,x}_{\tau_1}}_{\tau_2}) \mathbf{1}_{\{\tau_1 < t\}}
		\Big| \Fc_{\tau_1} \Big]
		+ \delta^nU(s^n_j,B^{0,x}_{\tau_1}) \mathbf{1}_{\{\tau_1<t\}}
	\bigg] \\
	&=& \sup_{(\tau_1, \tau_2) \in \Tc^2_{0,t}}  
	\EE \bigg[
		u^n(s^n_{j-2}, t- \tau_2, B^{0,x}_{\tau_2})
		+\sum_{k=j-1}^j \delta^nU(s^n_k, B^{0,x}_{\tau_{j-k+1}}) \mathbf{1}_{\{\tau_{j-k+1}<t\}}
	\bigg].
\end{eqnarray*}
	To conclude, it is enough to apply the same argument to iterate and to use the fact that 
	$u^n(0,t,x) = U(0,x)$ for any $t$ and $x$.
\end{proof}

	\begin{remark} \label{rem:Acn}
		For later use,  we also observe that 
		it is not necessary to restrict the stopping times w.r.t. the Brownian filtration,
		in the optimal stopping problem \eqref{eq:OS_discrTime_iterated}.
		In fact, one can consider a larger filtration with respect to which $B$ is still a Brownian motion.
		More precisely, let $\Ac^{n,0}_{t}$ denote the collection of all stopping rules 
		$$\alpha = (\Om^{\alpha}, \Fc^{\alpha}, \F^{\alpha}, \PP^{\alpha}, B^{\alpha} , \delta_0, \tau^{\alpha}_j, j=1, \cdots n)$$ 
		such that
		$ (\Om^{\alpha}, \Fc^{\alpha}, \F^{\alpha}, \PP^{\alpha})$ is a filtered probability space equipped with a standard Brownian motion $B^{\alpha}$ and  $(\tau^{\alpha}_j)_{j=1, \cdots, n}$ is a sequence of stopping times satisfying
		$0 \le \tau^{\alpha}_1 \le \ldots \le \tau^{\alpha}_n \le t$.
		Then one has
		$$
			u^n(s^n_j, t,x) = \sup_{\alpha \in \Ac^{n,0}_{t}}
			\EE^{\PP^{\alpha}} \Big[U(0,x+B^{\alpha}_{\tau^{\alpha}_j }) +\sum_{k=1}^j \delta^nU(s^n_k,x+B^{\alpha}_{\tau^{\alpha}_{j-k+1}}) \mathbf{1}_{\{\tau^{\alpha}_{j-k+1}<t\}}\Big].
		$$
		This equivalence is standard and very well known in case $n=1$, see also Lemma 4.9 of \cite{GTT_SEP} for the multiple stopping problem where $n \ge 1$.
	\end{remark}

\subsection{The Root solution given full marginals (Theorem \ref{thm:KTTintro}.$\mathrm{(i)}$)}
\label{subsec:cvg}

	By the same argument as in K{\"a}llblad, Tan and Touzi \cite{KTT}, the sequence of Root stopping times $(\sigma^n_1, \cdots, \sigma^n_n)_{n \ge 1}$ is tight in some sense
	and any limit provides an embedding solution given full marginals.
	
	\vspace{0.5em}
	
	More precisely, let us consider $n$-marginals $(\mu_{s^n_k})_{k =1 ,\cdots, n}$,
	and $(\sigma^n_k)_{k =1, \cdots, n}$ be the Root embedding defined in \eqref{eq:nRoot},
	we define $\alpha^*_n$ by \eqref{eq:nRoot_conti} as a $(\mu, \pi_n)$-embedding in sense of Definition \ref{def:embedding}.
	Notice that $\Pb^n := \PP^{\alpha^*_n} \circ (B^{\alpha^*_n}_{\cdot}, \tau^{\alpha^*_n}_{\cdot})^{-1}$ is a probability measure on the Polish space $C(\R_+, \R) \x \A([0,1], \R_+)$.
	This allows us to consider the weak convergence of the sequence $(\alpha^*_n)_{n \ge 1}$.
	Theorem \ref{thm:KTTintro} is then a consequence of the following convergence theorem, 
	which can be gathered from several results in \cite{KTT}.
	Recall also that $\Ac(\mu)$ and $\Ac(\mu, \pi_n)$ are defined in Definition \ref{def:embedding}.

	\begin{proposition} \label{prop:KTT}
	Let $(\pi_n)_{n \ge 1}$ be a sequence of partitions of $[0,1]$ with mesh $|\pi_n| \to 0$, and let $\alpha^*_n$ be the corresponding multiple-marginals Root embedding \eqref{eq:nRoot_conti}. Then, the sequence $\big( B^{\alpha^*_n}_{\cdot}, \tau^{\alpha^*_n}_{\cdot} \big)$ is tight, and any limit $\alpha^*$ is a full marginals embedding, i.e. $\alpha^* \in \Ac(\mu)$, with
		$$
			\PP^{\alpha^*_{n_k}} \circ (B^{\alpha^*_{n_k}}_{\tau^{\alpha^*_{n_k}}_s})^{-1} \longrightarrow \PP^{\alpha^*} \circ (B^{\alpha^*}_{\tau^{\alpha^*}_s})^{-1},
			~~\mbox{for all}~~
			s \in [0,1] \setminus \T,
		$$
	for some countable set $\T \subset [0, 1)$, and some subsequence $(n_k)_{k \ge 1}$, and 
	$$
			\EE^{\PP^{\alpha^*}} \Big[ \int_0^{\tau^{\alpha^*}_1} f(t) dt \Big]
			~=~
			\inf_{\alpha \in \Ac(\mu)} \EE^{\PP^{\alpha}} \Big[  \int_0^{\tau^{\alpha}_1} f(t) dt \Big],
		$$
for any non-decreasing and non-negative function $f: \R_+ \to \R_+$.
			\end{proposition}
	\begin{proof}
		$\mathrm{(i)}$ The first item is a direct consequence of Lemma 4.5 of \cite{KTT}.
		
		\noindent $\mathrm{(ii)}$ For the second item, we notice that 
		$\Phi(\om_\cdot,\theta_\cdot) := -\int_0^{\theta_1} f(t) dt$ is a continuous function defined on 
		$C(\R_+, \R) \x \A([0,1], \R_+)$ and bounded from above.
		Then it is enough to apply Theorem 2.3  and Proposition 3.6 of \cite{KTT} to obtain the optimality of $\alpha^*$.
	\end{proof}

\section{Proof of Theorems  \ref{th:OSP} and \ref{thm:KTTintro}.$\mathrm{(ii)}$}
\label{sec:proofs}

	Recall that $u(s,t,\cdot)$ is defined  in \eqref{eq:def_u} as the potential function of $B^{\alpha^*}_{t\wedge \tau^{\alpha^*}_s}$
	for an arbitrary Root solution $\alpha^*$ given full marginals.
	We provide an optimal stopping problem characterization as well as a PDE characterization for the function $u$ under Assumption \ref{assum:U}, which provides a proof of Theorem \ref{th:OSP}.
	Further, the uniqueness of the solution to the PDE induces the uniqueness result in part $\mathrm{(ii)}$ of Theorem  \ref{thm:KTTintro}.

\subsection{Characterization of $u$ by an optimal stopping problem (Theorem \ref{th:OSP}.$\mathrm{(i)}$)}

	By a slight abuse of notation, we can extend the definition of $u^n$ given in \eqref{eq:OS_discrTime} to $[0,1] \x \R_+ \x \R$ by setting
	$$
		u^n(s,t,x) := u^n(s^n_j, t,x) ~~~\mbox{whenever}~s \in (s^n_{j-1}, s^n_j].
	$$
	With $\Ac_t$ in Definition \ref{def:embedding}, we also define $\widetilde{u}$ as a mapping from $[0,1]\times\R_+\times\R$ to $\R$ by
	\begin{equation}\label{eq:def_u_intro}
		\widetilde{u}(s,t,x) 
		~:=~ 
		\sup_{\alpha \in \Ac_t^0} 
		\EE^{\PP^{\alpha}} \left[ 
			U(0,x+B^{\alpha}_{\tau^{\alpha}_s}) + \int_0^s \partial_s U(s-k,x+B^{\alpha}_{\tau^{\alpha}_k}) \mathbf{1}_{\{\tau^{\alpha}_k<t\}} dk 
		\right].
	\end{equation}
	The main objective of this section is to provide some characterisation of this limit law as well as the limit problem of $u^n$ \eqref{eq:OS_discrTime} used in the construction of the Root solution.

	\begin{proposition}\label{prop:conv_un}
		For all $(s,t,x)$, one has $u^n(s,t,x) \to \widetilde{u}(s,t,x)$ as $n \to \infty$.
	\end{proposition}
	\begin{proof}
	We start by rewriting the representation formula of $u^n(s^n_j,t,x)$ in Remark \ref{rem:Acn} as
		$$
			u^n(s^n_j, t,x) = \sup_{\alpha \in \Ac^{n,0}_t}
			\EE^{\PP^{\alpha}} \left[U(0,x+B^{\alpha}_{\tau^{\alpha}_j }) +\sum_{i=1}^j \delta^nU(s^n_{j-i+1},x+B^{\alpha}_{\tau^{\alpha}_i}) \mathbf{1}_{\{\tau^{\alpha}_i<t\}}\right].
		$$

	\noindent $\mathrm{(i)}$ First, for a fixed $n \in \N$, one can see $\Ac^{n,0}_t$ as a subset of $\Ac_t^0$ in the following sense.
	Given $\alpha \in \Ac^{n,0}_t$, and assume that $s \in (s^n_{j-1}, s^n_j]$.
	Notice that $U(s^n_j, y) - U(s, y) \le 0$ for all $y \in \R$,
	then it follows by direct computation that
	\begin{eqnarray*}
		&&
		\sum_{i=1}^j \delta^nU(s^n_{j-i+1},x+B^{\alpha}_{\tau^{\alpha}_i}) \mathbf{1}_{\{\tau^{\alpha}_i<t\}}\\
		&\le&
		\Big(U(s, x+ B^{\alpha}_{\tau^{\alpha}_1}) - U(s^n_{j-1}, x+ B^{\alpha}_{\tau^{\alpha}_1}) \Big) \mathbf{1}_{\{\tau^{\alpha}_1<t\}}
		+
		\sum_{i=2}^j \delta^nU(s^n_{j-i+1},x+B^{\alpha}_{\tau^{\alpha}_i}) \mathbf{1}_{\{\tau^{\alpha}_i<t\}}	
		\\
		&=& 
		\int_{s^n_{j-1}}^s  \partial_s U(k, x+ B^{\alpha}_{\tau^{\alpha}_1}) \mathbf{1}_{\{\tau^{\alpha}_1<t\}} dk
		+
		\sum_{i=2}^j \int_{s^n_{j-i}}^{s^n_{j-i+1}} \partial_s U(k, x+B^{\alpha}_{\tau^{\alpha}_i}) \mathbf{1}_{\{\tau^{\alpha}_i<t\}} dk \\
		&=&
		\int_0^{s-s^n_{j-1}} \partial_s U(s-k, x+ B^{\alpha}_{\tau^{\alpha}_1}) \mathbf{1}_{\{\tau^{\alpha}_1<t\}} dk
		+
		\sum_{i=2}^j \int_{s- s^n_{j-i+1}}^{s-s^n_{j-i}} \partial_s U(s-k, x+B^{\alpha}_{\tau^{\alpha}_i}) \mathbf{1}_{\{\tau^{\alpha}_i<t\}} dk
		\\
		&=&
		\int_0^{s} \partial_s U(s -k, x+B^{\alpha}_{\hat \tau^{\alpha}_k}) \mathbf{1}_{\{\hat \tau^{\alpha}_k<t\}} dk,
	\end{eqnarray*}
	where the last equality follows by setting $\hat \tau^{\alpha}_k := \tau^{\alpha}_i$ whenever $k \in [s-s^n_{j-i+1}, s- s^n_{j-i})$.
	By \eqref{eq:def_u_intro}, this implies that $u^n(s,t,x) =u^n(s^n_j, t, x) \le \widetilde{u}(s,t,x)$.

	\vspace{0.5em}

	\noindent $\mathrm{(ii)}$ Let $\alpha \in \Ac_t^0$, and define $\alpha_n \in \Ac^{n,0}_t$ by
	$$
		\tau^{\alpha_n}_j := \tau^{\alpha}_{s^n_j},
		~~\mbox{for}~j=1, \cdots,n.
	$$
	Let  $j_n$ be such that $s^n_{j_n}$ converges to $s$ as $n\rightarrow \infty$,
	then it follows that
	\begin{align*} 
		X_n ~&:=~ U(0,x+B^{\alpha}_{\tau^{\alpha_n}_{j_n}}) +\sum_{k=1}^{j_n} \delta^nU(s^n_{j_n-k+1},x+B^{\alpha}_{\tau^{\alpha_n}_k}) \mathbf{1}_{\{\tau^{\alpha_n}_k<t\}} \\
		&\underset{n\rightarrow \infty}{\longrightarrow}
		U(0,x+B^{\alpha}_{\tau^{\alpha}_s}) + \int_0^s \partial_s U(s-k,x+B^{\alpha}_{\tau^{\alpha}_k}) \mathbf{1}_{\{\tau^{\alpha}_k<t\}} dk ,
		~~\mbox{a.s.}
	\end{align*}
	Recall that by Assumption \ref{assum:U}, there exists some $C>0$ and $p>0$ such that
	$|U(0,x)| \le C+ |x|$ and $\sup_{s\in[0,1]} |\partial_s U(s,x)|\leq C(1+|x|^p)$ for any $x\in\R$,
	then $(X_n)_{n \ge 1}$ is in fact uniformly integrable. Hence for $\varepsilon>0$ such that $\alpha$ is $\varepsilon$-optimal in \eqref{eq:def_u_intro}, it follows from the previous remark and from Fatou's lemma that $\liminf_{n\rightarrow \infty} \EE[X_n] \geq \widetilde{u}(s,t,x)-\varepsilon$. Thus, 
	$
		\lim_{n \to \infty} u^n(s,t,x) \ge \widetilde{u}(s,t,x).
	$
	Therefore we have proved that ${\lim_{n \to \infty} u^n(s,t,x) = \widetilde{u}(s,t,x)}$.
\end{proof}

	\begin{lemma}\label{lem:u_BV}
		The function $\widetilde{u}(s,t,x)$ is non-increasing and Lipschitz in $s$ with a locally bounded Lipschitz constant $C(t,x)$,
		and is uniformly Lipschitz in $x$ and uniformly $\frac12$-H\"older in $t$.
	\end{lemma}
	\begin{proof}
	
	First, using representation formula of $u^n$ in \eqref{eq:RepresUn} and noticing that $y \mapsto |y-x|$ is convex,
	we see that $s \mapsto u^n(s, t,x)$ is non-increasing.
	Further, using \eqref{eq:OS_discrTime}, it follows immediately that
	$$
		u^n(s^n_j, t,x) - u^n(s^n_{j-1}, t,x) ~\ge~ U(s^n_j, x) - U(s^n_{j-1}, x).
	$$
	Then under Assumption \ref{assum:U}, one has $0 \ge \partial_s u^n(s, t,x) \ge - C(1+|x|^p)$ for some constant $C > 0$ and $p > 0$ independent of $n$. 
	By the limit result $u^n \to \widetilde u$,
	it follows that $\widetilde{u}(s,t,x)$ is non-increasing and Lipschitz in $s$ with a locally bounded Lipschitz constant $C(t,x)$.
	
	Finally, using again the representation formula of $u^n$ in \eqref{eq:RepresUn}, it is easy to deduce that $u^n(k, t,x)$ is uniformly Lipschitz in $x$ and $1/2$-H\"older in $t$, uniformly in $n$.
	As limit of $u^n$, it follows that $\widetilde{u}$ is also uniformly Lipschitz in $x$ and $1/2$-H\"older in $t$.
\end{proof}

	We next show that the function $u$ defined by \eqref{eq:def_u} is also the limit of $u^n$,
	which leads to the equivalence of $u$ and $\widetilde u$, and then Theorem \ref{th:OSP}.$\mathrm{(i)}$ readily follows.
	\begin{proposition} \label{prop:equiv_u}
		For all $(s,t,x) \in [0,1] \x \R_+ \x \R$, one has $u(s,t,x) = \widetilde{u}(s,t,x)$.
	\end{proposition}
\begin{proof}
	By Theorem 3.1 of \cite{COT} together with our extended definition in \eqref{eq:nRoot_conti}, $
		u^n(s, t, x) 
		=
		- \mathbb{E}^{\PP^{\alpha^*_{n}}} \Big[ \Big|B^{\alpha^*_{n}}_{t \wedge \tau^{\alpha^*_{n}}_s} - x \Big| \Big].
	$
	Moreover, by Proposition \ref{prop:KTT},
	there exists a countable set $\mathbb{T} \subset [0,1)$ and a subsequence $(n_k)_{k \ge 1}$ such that
	$\PP^{\alpha^*_{n_k}} \circ (B^{\alpha^*_{n_k}}_{\tau^{\alpha^*_{n_k}}_s})^{-1} 
		\to \PP^{\alpha^*} \circ (B^{\alpha^*}_{\tau^{\alpha^*}_s})^{-1} $.
	As $\mathbb{E}[\max_{0 \le r \le t} |B_r -x| \big] < \infty$ for a Brownian motion, it follows that $
		\mathbb{E}^{\PP^{\alpha^*_{n_k}}} \big[ \big|B^{\alpha^*_{n_k}}_{t \wedge \tau^{\alpha^*_{n_k}}_s} - x \big| \big] 
		~\longrightarrow~
		\mathbb{E}^{\PP^{\alpha^*}} \big[ \big|B^{\alpha^*}_{t \wedge \tau^{\alpha^*}_s} - x \big| \big]$, for all $s \in [0,1] \setminus \mathbb{T}$. Hence,
	$$
		u^n(s,t,x) \longrightarrow u(s,t,x),
		~~\mbox{for all}~
		s \in [0,1] \setminus \mathbb{T}.
	$$
	Further, by the right-continuity of $s \mapsto \tau^{\alpha^*}_s$,
	it is easy to deduce that $s \mapsto u(s,t,x) := - \mathbb{E}^{\PP^{\alpha^*}} |B^{\alpha^*}_{t \wedge \tau^{\alpha^*}_s} - x|$ is right-continuous.

	On the other hand, we know from Proposition \ref{prop:conv_un} that $u^n(s,t,x) \to \widetilde{u}(s,t,x)$,
	and from Lemma \ref{lem:u_BV} that $\widetilde{u}$ is a continuous function in all arguments,
	it follows that
	$u(s,t,x) = \widetilde{u}(s,t,x)$ holds for all $(s,t,x) \in [0,1] \x \R_+ \x \R$.
\end{proof}

\begin{remark}
	Formally, we can understand the above result in the following equivalent way.
	The $n$-marginals Root solution $(\sigma^n, B)$ converges weakly to a full marginal Root solution $(\sigma, B)$
	which then satisfies:
	$$
		u(s,t,x) = -\EE^{\mu_0} \big[ \big| B_{t\wedge \sigma_s} -x \big| \big].
	$$
\end{remark}

\subsection{PDE characterization of $u$}



	We now provide the
	
	\noindent  {\bf Proof of Theorem \ref{th:OSP}}.$\mathrm{(ii)}$.
	{\it Step 1.}
	We first notice that the continuity of $u(s,t,x)$ in $(s,t,x)$ follows directly by Lemma \ref{lem:u_BV} and Proposition \ref{prop:equiv_u}.
	
	\vspace{0.5em}
	
	\noindent {\it Step 2.}
	In a filtered probability space $(\Om, \Fc, \F, \PP)$ equipped with a Brownian motion $W$, we denote by $\Uc_t$
	the collection of all $\F$-predictable processes $\gamma = (\gamma_r)_{r \ge 0}$ such that $\int_0^1 \gamma_r^2 dr \le t$.
	Given a control process $\gamma$, we define two controlled processes $X^{\gamma}$ and $Y^{\gamma}$ by
	\begin{align*}
		X_s^{\gamma} := x + \int_0^s \gamma_r dW_r,
		~~~
		Y_s^{\gamma} := \int_0^s \gamma^2_r dr.
	\end{align*}
	By a time change argument, one can show that
	\begin{equation} \label{eq:represent_u_ctrl}
		u(s,t,x) 
		= 
		\sup_{\gamma \in \Uc_t} 
		\mathbb{E} \Big[ U(0, X_s^{\gamma}) + \int_0^s \partial_s U(s-k, X_k^{\gamma}) \mathbf{1}_{\{Y^{\gamma}_k < t\}} dk \Big].
	\end{equation}
	Indeed, given $\gamma \in \Uc_t$, one obtains a square integrable martingale $X^{\gamma}$ which has the representation $X^{\gamma}_s = W_{Y^{\gamma}_s}$, where $W$ is a Brownian motion and $Y^{\gamma}_s$ are all stopping times, and it induces a stopping rule in $\Ac_t^0$.
	By \eqref{eq:def_u_intro} and Proposition \ref{prop:equiv_u}, it follows that in \eqref{eq:represent_u_ctrl}, the left-hand side is larger than the right-hand side.
	On the other hand, given an increasing sequence of stopping times $(\tau_1, \cdots, \tau_n) \in \Tc^n_{0,t}$, 
	we define $\Gamma_s := \tau_j \vee \frac{s - s^n_j}{s^n_{j+1} - s} \wedge \tau_{j+1}$ for all $s \in [s^n_j, s^n_{j+1})$.
	Notice that $s \mapsto \Gamma_s$ is absolutely continuous and $\Gamma_{s^n_j} = \tau_j$ for all $j =1 ,\cdots, n$.
	Then one can construct a predictable process $\gamma$ such that 
	$$
		\PP_0 \circ \Big(\int_0^{\cdot} \gamma_r dB_r,  \int_0^{\cdot} \gamma_r^2 dt \Big)^{-1}
		=
		\PP_0 \circ \Big(B_{\Gamma_{\cdot}},  \Gamma_{\cdot} \Big)^{-1}.
	$$
	Using the definition of $u^n$ in \eqref{eq:OS_discrTime_iterated} and its convergence in Proposition \ref{prop:conv_un}, one obtains that in \eqref{eq:represent_u_ctrl}, the right-hand side is larger than the left-hand side.

	\vspace{0.5em}

	The above optimal control problem satisfies the dynamic programming principle (see e.g. \cite{ElKarouiTan}): for any family of stopping times $(\tau^{\gamma})_{\gamma \in \Uc_t}$ dominated by $s$, one has
	\begin{equation} \label{eq:DPP}
		u(s,t,x) 
		= 
		\sup_{\gamma \in \Uc_t} 
		\mathbb{E} 
		\Big[
			u \big(s- \tau^{\gamma}, t- Y^{\gamma}_{\tau^{\gamma}}, X^{\gamma}_{\tau^{\gamma}} \big)
			+
			\int_0^{\tau^{\gamma}} \partial_s U(s-k, X^{\gamma}_{\tau^{\gamma}}\big) \mathbf{1}_{\{Y^{\gamma}_k < t\}} dk
		\Big].
	\end{equation}
	
	\noindent {\it Step 3 (supersolution).} Let $z=(s,t,x)\in{\rm int}^p(\mathbf{Z})$ be fixed,
	and $\varphi\in C^2(\mathbf{Z})$ be such that $0 = (u - \varphi)(z) = \min_{z'\in \mathbf{Z}}(u-\varphi)$. By a slight abuse of notation, denote by $\partial_s U(z')$ the quantity $\partial_s U(s',x')$ for any $z' = (s',t',x')$. Then by \eqref{eq:DPP}, for any family of stopping times $(\tau^{\gamma})_{\gamma\in \Uc_t}$  dominated by $s$, one has, 
	$$
		\sup_{\gamma \in \Uc_t} \mathbb{E} 
		\Big[ 
			\int_0^{\tau^{\gamma}}  \big( -  \partial_s \varphi(Z_k) + \partial_s U(Z_k) \mathbf{1}_{\{Y^\gamma_k<t\}} \big) dk
			+
			\int_0^{\tau^{\gamma}}  \gamma_k^2 \big(- \partial_t \varphi + \frac12 \partial^2_{xx} \varphi)(Z_k) dk
		\Big] 
		\le 0,
	$$
	where $Z_k := (s-k, t - Y^{\gamma}_k, X^{\gamma}_k)$ and $X^\gamma_0= x$.
	Choosing $\gamma_\cdot \equiv 0$ and $\tau^{\gamma} \equiv h$ for $h$ small enough, we get
	$$
		\big(- \partial_s \varphi + \partial_s U \big) (z) \le 0.
	$$
	On the other hand, choosing $\gamma_\cdot \equiv \gamma_0$ for some constant $\gamma_0$ and $\tau^{\gamma} := \inf\{ k\ge 0~: |X^{\gamma}_k -x| + |Y^{\gamma}_k| \ge h\}$, then by letting $\gamma_0$ be large enough and $h$ be small enough, one can deduce that 
	$$
		\big(- \partial_t \varphi + \frac12 \partial^2_{xx} \varphi \big)(s,t,x) \le 0.
	$$

	\noindent {\it Step 4 (subsolution).} Assume that $u$ is not a viscosity sub-solution, then there exists $z=(s,t,x)\in\mbox{int}^p(\mathbf{Z})$ and $\varphi\in C^2(\mathbf{Z})$, such that $0 = (u-\varphi)(z) = \max_{z' \in \mathbf{Z}} (u-\varphi)(z')$, and
	$$
		 \min \big\{ \partial_t \varphi - \frac12 \partial^2_{xx} \varphi,~ \partial_s(\varphi-U) \big\} (s,t,x) > 0.
	$$
	By continuity of $U$ and $\varphi$, we may find $R>0$ such that 
	\begin{equation} \label{eq:super_sol_contrad}
		\min  \big\{ \partial_t \varphi - \frac12 \partial^2_{xx} \varphi,~ \partial_s(\varphi-U) \big\} \ge 0,
		~~\mbox{on}~~
		B_R(z), 
	\end{equation}
where $B_R(z)$ is the open ball with radius $R$ and center $z$.
	Let $\tau^{\gamma} := \inf\{k ~: Z^{\gamma}_k \notin B_R(z) \text{ or } Y^\gamma_k\geq t\} $, 
	and notice that $\max_{\partial B_R(s,t,x)} (u-\varphi) = - \eta < 0$, by the strict maximality property. Then it follows  from \eqref{eq:DPP} that
	\begin{eqnarray*}
		0 
		&=&
		\sup_{\gamma} \mathbb{E} \Big[ 
			u \big(s- \tau^{\gamma}, t- Y^{\gamma}_{\tau^{\gamma}}, X^{\gamma}_{\tau^{\gamma}} \big) - u(s,t,x)
			+
			\int_0^{\tau^{\gamma}} \partial_s U(s-k, X^{\gamma}_{\tau^{\gamma}}\big) \mathbf{1}_{\{Y^{\gamma}_k < t\}} dk 
		\Big] \\
		&\le&
		-\eta + \sup_{\gamma}  \mathbb{E}\Big[
			\int_0^{\tau^{\gamma}}  \Big( - \partial_s (\varphi - U \mathbf{1}_{\{Y^\gamma_k<t\}}) - (\partial_t \varphi - \frac12 \partial^2_{xx}\varphi) \gamma^2_k\Big)(Z_k) dk
		\Big]
		~\le~
		-\eta,
	\end{eqnarray*}
	where the last inequality follows by \eqref{eq:super_sol_contrad}. This is the required contradiction.
	\qed

\subsection{Comparison principle of the PDE (Theorems \ref{th:OSP}.$\mathrm{(iii)}$ and \ref{thm:KTTintro}.$\mathrm{(ii)}$)}

	Recall that the operator $F$ is defined in \eqref{eq:def_F} and we will study the PDE \eqref{eq:PDE}.
	For any $\eta \ge 0$, a lower semicontinuous function $w: \mathbf{Z} \to \R$ is called an $\eta$-strict viscosity supersolution of \eqref{eq:PDE} if 
		$w|_{\partial^p\mathbf{Z}} \ge \eta+ U(0, \cdot)$ and $F(D\varphi, D^2 \varphi)(z_0) \ge \eta$ for all $(z_0, \varphi) \in \mbox{int}^p(\mathbf{Z}) \x C^2(\mathbf{Z})$ satisfying $(w-\varphi)(z_0) = \min_{z \in\mathbf{Z}} (w - \varphi)(z)$.

\begin{proposition}[Comparison]
\label{prop:comparison}
	Let $v$ (resp. $w$) be an upper (resp. lower) semicontinuous viscosity subsolution (resp. supersolution) of the equation \eqref{eq:PDE} satisfying
	\begin{eqnarray*}
		v(z) \;\le\; C(1+t+|x|)
		~\mbox{and}~
		w(z)\;\ge\;  - C(1+t+|x|),
		~z\in\mathbf{Z},
		&~\mbox{for some constant}~&
		C > 0.
	\end{eqnarray*}
	Then $v\le w$ on $\mathbf{Z}$.
\end{proposition}

\proof We proceed in three steps.

	\vspace{0.5em}

\noindent $\mathrm{(i)}$ In this step, we prove the result under the assumption that the comparison result holds true if the supersolution is $\eta-$strict for some $\eta>0$. First, direct verification reveals that the function:
 \begin{eqnarray*}
 w^1(s,t,x)
 :=
 U(0,x)+\eta (1+s+t),
 &\quad (s,t,x)\in\mathbf{Z},&
 \end{eqnarray*}
is an $\eta-$strict supersolution. For all $\mu\in(0,1)$, we claim that the function $w^\mu:=(1-\mu)w+\mu w^1$ is a $\mu\eta-$strict viscosity supersolution. Indeed, this follows from the proof of Lemma A.3 (p.52) of Barles and Jakobsen \cite{barles2002convergence}, which shows that $w^\mu$ is a viscosity supersolution of both linear equations:
 \begin{eqnarray*}
 \partial_t w^\mu - \frac12\partial^2_{xx}w^\mu \ge \mu\eta
 &\quad \mbox{and} \quad &
 \partial_s w^\mu - \partial_sU \ge \mu\eta.
  \end{eqnarray*} 
	Assume that the comparison principle holds true if the supersolution is strict,
	then it follows that $v\leq w^\mu$ on $\mathbf{Z}$.
	Let $\mu\searrow 0$, we obtain $v\leq w$ on $\mathbf{Z}$.

	\vspace{0.5em}

\noindent $\mathrm{(ii)}$ In view of the previous step, we may assume without loss of generality that $w$ is an $\eta-$strict supersolution. In order to prove the comparison result in this setting, we assume to the contrary that 
  \begin{eqnarray} \label{contra-z}
 \delta \;:=\; (v-w)(\hat z) \;>\; 0,
 &\quad \mbox{for some}~&
 \hat z\in\mathbf{Z},
 \end{eqnarray}
and we work toward a contradiction. Following the standard doubling variables technique, we introduce for arbitrary $\alpha,\eps>0$:
 \begin{eqnarray*}
 \Phi^{\alpha,\eps}(z,z')
 :=
 \frac{\alpha}{2}\big|z-z'\big|^2
 +\eps\big(\varphi(z)+\varphi(z')\big),
 ~~\mbox{with}~~
 \varphi(z)
 :=
 \frac{1}{2}\big(t^2+x^2\big),
 &~z,z'\in\mathbf{Z},&
 \end{eqnarray*}
and the corresponding maximum
  \begin{eqnarray} \label{Malphaeps}
 M^{\alpha,\eps}
 :=
 \sup_{(z,z')\in\mathbf{Z}\times\mathbf{Z}} 
 \big\{v(z)-w(z')-\Phi^{\alpha,\eps}(z,z')\big\}
 &\ge&
 \delta-2\eps\varphi(\hat z)
 \;>\;
 0,
 \end{eqnarray}
by \eqref{contra-z}, for sufficiently small $\eps>0$. 
	By the bounds on $v$ and $w$, it follows that the above supremum may be confined to a compact subset of $\mathbf{Z}\times\mathbf{Z}$. Then the upper semicontinuity of the objective function implies the existence of a minimizer $(z^{\alpha,\eps},{z'}^{\alpha,\eps})\in\mathbf{Z}\times\mathbf{Z}$, i.e.
 \begin{eqnarray*}
 M^{\alpha,\eps}
 &=&
 v\big(z^{\alpha,\eps}\big)-w\big({z'}^{\alpha,\eps}\big)
 -\frac{\alpha}{2}\big|z^{\alpha,\eps}-{z'}^{\alpha,\eps}\big|^2
 -\eps\big(\varphi(z^{\alpha,\eps})+\varphi({z'}^{\alpha,\eps})\big),
 \end{eqnarray*}
and there exists a converging subsequence $\big(z_n^\eps,{z'}_n^\eps\big):=\big(z^{\alpha_n,\eps},{z'}^{\alpha_n,\eps}\big)\longrightarrow (z^\eps, {z'}^\eps)\in\mathbf{Z}\times\mathbf{Z}$, for some $(\alpha_n)_n$ converging to $\infty$. Moreover, denoting by $z^*$ any minimizer of $v-w-2\eps\varphi$, we obtain from the inequality $(v-w-2\eps\varphi)(z^*)\le M^{\alpha_n,\eps}$ that
 \begin{eqnarray*}
 \ell
 &:=&
 \limsup_{n\to\infty} \frac{\alpha_n}{2}\big| z_n^\eps - {z'}_n^\eps \big|^2 \\
 &\le&
 \limsup_{n\to\infty} 
 v(z_n^\eps)-w({z'}_n^\eps) 
 - \eps\big(\varphi(z_n^\eps)+\varphi({z'}_n^\eps) \big)
 - (v-w-2\eps\varphi)(z^*)
 \\
 &\le&
 v(z^\eps)-w({z'}^\eps)
 - \eps\big(\varphi(z^\eps)+\varphi({z'}^\eps) \big)
 - (v-w-2\eps\varphi)(z^*)
 \;<\; \infty.
 \end{eqnarray*}
Then $z^\eps={z'}^\eps$, and $0\le\ell\le (v-w-2\eps\varphi)(z^\eps)- (v-w-2\eps\varphi)(z^*)\le 0$ by the definition of $z^*$. Consequently: 
  \begin{equation} \label{step4}
	\begin{split}
 &z^\eps \;=\; {z'}^\eps,~
 \alpha_n\big| z_n^\eps - {z'}_n^\eps \big|^2 \longrightarrow 0, ~\mbox{and}~~\\
 &M^{\alpha_n,\eps}\longrightarrow\sup_\mathbf{Z}\left\{v-w-2\eps\varphi\right\},
 ~\mbox{as}~~n\to\infty.
	\end{split}
 \end{equation}
Finally, as $v$ is a subsolution and $w$ a supersolution, we see that if $z^\eps$ lies in $\partial^p \mathbf{Z}$, we would have $\limsup_{n\to\infty} M^{\alpha_n,\eps}\le - \eta -2\eps\varphi\big(z^\eps\big) < 0$, which is in contradiction with the positive lower bound in \eqref{Malphaeps}. 
	Consequently $z^\eps$ is in the parabolic interior $\mbox{\rm int}^p (\mathbf{Z})$ of $\mathbf{Z}$, 
	and therefore both $z^\eps_n$ and ${z'}^\eps_n$ are in $\mbox{\rm int}^p (\mathbf{Z})$  for sufficiently large $n$.
 
 	\vspace{0.5em}

	\noindent $\mathrm{(iii)}$ 
	We now use the viscosity properties of $v$ and $w$ at the interior points $z^\eps_n$ and ${z'}^\eps_n$, for large $n$. 
	By the Crandall-Ishii Lemma, see e.g. Crandall, Ishii and Lions \cite[Theorem 3.2]{crandall1992user}, we may find for each such $n$ two pairs $(p^\eps_n,A^\eps_n)$ and $(q^\eps_n,{B}^\eps_n)$ in $\R^3\times\mathbb{S}_3$, such that 
	$$
		\big(p^\eps_n+ \eps D \varphi(z^{\eps}_n), A^\eps_n+ \eps D^2 \varphi(z^{\eps}_n)\big)
		~\in~  \overline{J}v(z^{\eps}_n),
	$$
	$$
		\big(q^\eps_n - \eps  D \varphi({z'}^{\eps}_n),B^\eps_n- \eps D^2 \varphi({z'}^{\eps}_n)\big)
		~\in~  \underline{J}w({z'}^{\eps}_n),
	$$
	$$
		p^\eps_n=q^\eps_n=\alpha_n(z^\eps_n-{z'}^\eps_n)
		~~\mbox{and}~~
		A^\eps_n\le B^\eps_n,
  	$$
	where $\overline J$ and $\underline J$ denote the second order super and subjets, see \cite{crandall1992user}.
 Then, it follows from the subsolution property of $v$ and the $\eta-$strict supersolution of $w$ that
\begin{eqnarray*}
&&
	\min\Big\{ \alpha_n(t^\eps_n-{t'}^\eps_n) + \eps t^\eps_n
                  -\frac12 (A^\eps_{3,3,n}+\eps) ,
                  \alpha_n(s^\eps_n-{s'}^\eps_n) - \partial_s U(s^\eps_n,x^\eps_n)
	\Big\} 
\\
&&
	\!\!\le 0 \le
 -\eta 
 + \min\Big\{ \alpha_n(t^\eps_n-{t'}^\eps_n) - \eps {t'}^\eps_n
                  -\frac12 (B^\eps_{3,3,n}-\eps) ,
                  \alpha_n(s^\eps_n-{s'}^\eps_n) - \partial_s U({s'}^\eps_n,{x'}^\eps_n)
       \Big\}  
 \\
&&
 \hspace{5mm}\le
 -\eta 
 + \min\Big\{ \alpha_n(t^\eps_n-{t'}^\eps_n) - \eps {t'}^\eps_n
                  -\frac12 (A^\eps_{3,3,n}-\eps) ,
                  \alpha_n(s^\eps_n-{s'}^\eps_n) - \partial_s U({s'}^\eps_n,{x'}^\eps_n)
       \Big\},
 \end{eqnarray*}
by the inequality $A^\eps_n\le B^\eps_n$. This implies that
 \begin{eqnarray*}
 0
 &\le&
 -\eta
 -\eps(t^\eps_n+{t'}^\eps_n) 
 +\eps
 +\big| \partial_s U({s}^\eps_n,{x}^\eps_n) - \partial_s U({s'}^\eps_n,{x'}^\eps_n) \big|
 \\
 &\le&
 \;\;
 -\eta
 +\eps
 +\big| \partial_s U({s}^\eps_n,{x}^\eps_n) - \partial_s U({s'}^\eps_n,{x'}^\eps_n) \big|
 \;\longrightarrow\;
 -\eta
 +\eps,
 ~~\mbox{as}~n\to\infty,
 \end{eqnarray*}
which provides the required contradiction for sufficiently small $\eps>0$. 
\qed

	\begin{remark}
		To conclude the proofs of Theorem \ref{th:OSP} $\mathrm{(iii)}$ as well as Theorem  \ref{thm:KTTintro} $\mathrm{(ii)}$, 
		we notice that 
		both potential functions $U$ and $U_{\mathbf{N}(0,1)}$  have linear growth in $x$, then the linear growth requirements of the comparison result in Proposition \ref{prop:comparison} follow from \eqref{eq:linear_growth},
		and uniqueness for the PDE \eqref{eq:PDE} in part $\mathrm{(iii)}$ of Theorem \ref{th:OSP} follows immediately.
		Further, this implies also the uniqueness of the potential functions of $B^{\alpha^*}_{\tau^{\alpha^*}_s \wedge t}$ for all $s \in [0,1]$ and $t \ge 0$,
		and hence the uniqueness of law of $B^{\alpha^*}_{\tau^{\alpha^*}_s \wedge t}$
		in part $\mathrm{(ii)}$ of Theorem \ref{thm:KTTintro}.	
	\end{remark}

\section{Further discussion and examples}
\label{sec:discussion}

\subsection{On Assumption \ref{assum:U}}
\label{subsec:Ass_U}

	Our main result (Theorem \ref{th:OSP}) relies crucially on the growth condition for $x \mapsto \sup_{s \in [0,1]} |\partial_s U(s,x)|$ in Assumption \ref{assum:U}.
	Let us provide here some examples and counter-examples  on this condition.
	
	\begin{example} \label{exam:assumU1}
		Let $X$ be a random variable in a probability space $(\Om, \Fc, \mathbb{P})$ such that $\mathbb{E}[|X|] < \infty$, $\mathbb{E}[X] = 0$, and $\mathbb{P}[X=x] = 0$ for all $x \in \R$.
		For some $s_0 \ge 0$, we define a family of marginals $(\mu_s)_{s \in [0,1]}$ by 
		$\mu_s  := \mathbb{P} \circ ( (s+s_0) X)^{-1}$.
		Then each marginal $\mu_s$ is centred and has finite first moment, and $s \mapsto \mu_s $ is increasing in convex order (see e.g. \cite[Chapter 1]{HPRY}). 
		In this case, one has
		\begin{align*}
			U(s, x) 
			:=
			-\mathbb{E} \big[ \big| (s+s_0)X - x \big| \big] 
			&=
			 -(s+s_0) \mathbb{E} \Big[ \Big| X - \frac{x}{s+s_0} \Big| \Big] \\
			 &=
			 -(s+s_0) U \Big(1-s_0, \frac{x}{s+s_0} \Big).
		\end{align*}
		Further, by direct computation,
		$$
			\partial_s U(s,x) 
			= 
			-U \Big(1-s_0, \frac{x}{s+s_0} \Big) +  \partial_x U \Big(1-s_0, \frac{x}{s+s_0} \Big) \frac{x}{s+ s_0},
		$$
		where $\partial_x U$ is bounded continuous as the distribution function $x \mapsto \mathbb{P}[X \le x]$ is bounded continuous.
		Assume in addition that $s_0>0$,
		then $\sup_{0 \le s \le 1} |\partial_s U(s,x)|$ has linear growth in $x$.
		In particular, Assumption \ref{assum:U} holds true.
	\end{example}
	
	If we set $s_0 = 0$ in Example \ref{exam:assumU1}, then it is clear that $\lim_{s \searrow 0} |\partial_s U(s,x)| = \infty$, i.e. there is a singularity/discontinuity for $\partial_s U(s,x)$ at $s=0$, and thus Assumption \ref{assum:U} fails.
	In the following, we also provide an example, where $\partial_s U$ is continuous on $[0,1] \x \R$,
	but $x \mapsto \sup_{0 \le s \le 1} |\partial_s U(s,x)|$ may have any growth.
	Before giving the other example, let us introduce a (variated potential) function $V:[0,1] \x \R \to \R$ for $(\mu_s)_{0 \le s \le 1}$ by
	\begin{equation} \label{eq:call_price}
		V(s, x) := \int_{\R} (y-x)_+ \mu_s(dy).
	\end{equation}
	When $\mu_s$ is centred, one has
	$$
		2 V(s,x)  = -U(s,x) -x
		~~\mbox{so that}~~
		\partial_s U(s,x) = -2 \partial_s V(s,x).
	$$
	Therefore, it is equivalent to consider the growth of $x \mapsto \sup_{s \in [0,1]}|\partial_s V(s,x)|$ and that of
	$x \mapsto \sup_{s \in [0,1]}|\partial_s U(s,x)|$.
	Further, Equality \eqref{eq:call_price} allows to define a family of marginals $(\mu_s)_{s \in [0,1]}$, increasing in convex order, from $V:[0,1] \x \R \to \R$ whenever it satisfies that, for all $s \in [0,1]$,
	$$
		x \mapsto V(s,x) ~\mbox{is non-increasing and convex, and}
		~\lim_{x \to \infty}\! V(s,x) = \lim_{x \to -\infty}\! V(s,x) + x = 0,
	$$
	and for all $x \in \R$,
	$$
		s \mapsto V(s,x) ~\mbox{is non-decreasing}. 
	$$
	We refer to \cite[Proposition 2.1]{HirschRoynette} for such a characterisation.

	\begin{example}
		Let $V(1,x)$ be the (variated) potential function of the standard normal distribution $\mathbf{N}(0,1)$, 
		defined by
		$$
			V(1,x) := \int_{\R} (y-x)_+ \frac{1}{\sqrt{2 \pi}} e^{-y^2/2} dy.
		$$
		It is clear that $V(1,x)$ is strictly positive, decreasing and convex on $\R$.
		Let $(t_j)_{j=0}^{\infty}$ be an increasing sequence of real numbers such that $t_0 = 0$ and $\lim_{j\to \infty} t_j = 1$. We then define a sequence of points $(x_j)_{ j \ge 1}$ and functions $x \mapsto V(t_j, x)$ by the following induction argument:
		Let $x_0 = 0$; given $x_j$, one defines
		$$
			x_{j+1} ~:=~ x_j - \frac{V(1,x_j)}{\partial_x V(1,x_j)},
		$$
		and
		$$
			V(t_j, x) 
			~:=~ 
			V(1,x) \mathbf{1}_{\{x \le x_j\}}
			+
			\big(V(1, x_j) + \partial_xV(1,x_j) (x - x_j) \big)_+ \mathbf{1}_{\{x > x_j\}}.
		$$
		Then for each $j \ge 0$, $x \mapsto V(t_j, x)$ is the (variated) potential function of some marginal distribution $\mu_{t_j}$ with the equality in \eqref{eq:call_price},
		and 
		$$
			j \mapsto \mu_{t_j}
			~\mbox{is increasing in convex order and}~
			~\mu_{t_j} \to \mu_1 := \mathbf{N}(0,1),
		$$
		as $V(t_j, x) \to V(1,x)$ when $j \to \infty$.
		Moreover, the two functions $V(t_j, x)$ and $V(t_{j+1},x)$ coincide on $\R \setminus [x_j, x_{j+2}]$.

		\vspace{0.5em}
		
		For any growth function $G: \R \to \R_+$, 
		we consider a sequence of smooth increasing functions $\phi_j : [t_j, t_{j+1}] \to [0,1]$, for $j \ge 0$, such that
		\begin{equation} \label{eq:growth_condition}
			\phi_j(t_j) = \phi_j ' (t_j) = \phi_j '(t_{j+1}) = 0,
			~
			 \phi_j(t_{j+1}) = 1,
			~\mbox{and}~
			\phi_j' \big( (t_j+t_{j+1})/2 \big) \ge \frac{G(x_{j+1})}{V(1, x_{j+1})}.
		\end{equation}
		We then define $x \mapsto V(s,x)$ (or equivalently $\mu_s$) for all $s \in (t_j, t_{j+1})$ and all $j \ge 0$ by
		$$
			V(s,x) ~:=~ (1-\phi_j(s)) V(t_j, x) + \phi_j(s) V(t_{j+1}, x).
		$$
		Then it is direct to check that the family of marginal distributions $(\mu_s)_{s \in [0,1]}$ is increasing in convex order, and $\partial_s V(s,x)$ exists and is continuous on $[0,1] \x \R$,
		and by \eqref{eq:growth_condition},
		$$
			\Big| \max_{s \in [0,1]} \partial_s V(s,x_j) \Big| \ge G(x_j),~\mbox{for all}~j \in \N.
		$$

		Finally, we notice that in the above example, the derivative $\partial^2_{xx} V(s,x)$ may not exists,
		but we can always consider the convolution of $x \mapsto V(s,x)$ with some smooth function in compact support, 
		so that $\partial^2_{xx} V(s,x)$ exists for all $x \in \R$, and it provides the density function of $\mu_s$.
	\end{example}

\subsection{Towards a hitting time characterization for the Root embedding}

	In the case with finitely many marginals  $(\mu_{s^n_j})_{j=1, \cdots, n}$,
	the Root stopping times $\{\sigma^n_j\}_{j=1,\cdots, n}$ 
	are defined successively as hitting times of barriers, that is
	\begin{equation} \label{eq:def_nRoot}
		\sigma^n_j := \inf \big\{ t \ge \sigma^n_{j-1} ~: (t, B_t) \in \Rc^n_j \big\},
	\end{equation}
	with barriers $\Rc^n_j$ defined by
	\begin{equation}  \label{eq:def_nRoot_barrier}
		\mathcal{R}_j^n = \{(t,x): ~\delta^n u_n(s^n_j,t,x) = \delta^n U(s^n_j,x)\} .
	\end{equation}
	In the full marginals case,
	it is natural to seek to characterize the limit Root solution $(\tRoot^{\infty}_s)_{s \in [0,1]}$ as hitting times,
	and a natural guess would be
	\begin{equation} \label{eq:faux_def_Root_full_marg}
		\tRoot^{\infty}_s ~=~ \inf\left\{ t\geq \tRoot^{\infty}_{s-}:~ (t,W_t) \in \mathcal{R}_s \right\},~\mbox{for all}~s \in [0,1],
	\end{equation}
	with a family of barriers $\mathcal{R} = \{\mathcal{R}_s\}_{s\in[0,1]}$.
	Moreover, a natural candidate for the family of the barriers would be,
	when the partial derivative $\partial_s u(s,t,x)$ is well defined,
	\begin{align*}
		\mathcal{R}_s ~:=~ \big\{(t,x):~ \partial_s u(s,t,x) = \partial_s U(s,x) \big\}.
	\end{align*}

	Our first observation is that \eqref{eq:faux_def_Root_full_marg} cannot be used
	as a definition for an uncountable family of stopping times. 
	Indeed, for the case with finite marginals as in \eqref{eq:def_nRoot}, the induction argument defines well the family $(\sigma^n_j)_{j = 1, \cdots, n}$.
	However, when the index set cannot be written as a countable and ordered sequence of  time instants,
	one cannot cover all time instants in the set by an induction argument as in \eqref{eq:faux_def_Root_full_marg}.

	\vspace{0.5em}
	
	Nevertheless, one can always define the family $(\tRoot^{\infty}_s)_{s \in [0,1]}$ 
	as limit of $\{\sigma^n_j\}_{j=1,\cdots, n}$ in terms of the convergence of stochastic (non-decreasing) processes (see Section \ref{subsec:cvg}),
	and then expect to use \eqref{eq:faux_def_Root_full_marg} as a characterization of $(\tRoot^{\infty}_s)_{s \in [0,1]}$.
	As $(\tRoot^{\infty}_s)_{s \in [0,1]}$ is obtained as weak limit of non-decreasing processes (w.r.t. the L\'evy metric)
	it is natural to study the convergence of the set valued process 
	$(\Rc^n_s)_{s \in [0,1]}$ defined by $\Rc^n_s := \Rc^n_j$ for $s \in [s^n_j, s^n_{j+1})$ (with a slight abuse of notation).
	The argument of the convergence of the barriers has been crucially used in the seminal paper \cite{Root} by considering an appropriate topology on the space of subsets of $[0, \infty] \x \R$ to construct a one-marginal Root solution.
	In our context, when the barriers  $(\Rc^n_s)_{s \in [0,1]}$ are ordered, the convergence of the set-valued process
	$(\Rc^n_s)_{s \in [0,1]}$  would be easy or even trivial (see some trivial examples in Example \ref{exam:RootOrdered}).
	When the barriers  $(\Rc^n_s)_{s \in [0,1]}$ are not ordered,
	one generally needs some regularity of the path $s \mapsto \Rc^n_s$ as required in the Arzel\`a-Ascoli theorem.
	Nevertheless, it seems unclear to us to obtain estimations on the path regularity of barriers from its definition in \eqref{eq:def_nRoot_barrier}.
	
	\begin{example}[Ordered Root solutions] \label{exam:RootOrdered}
		Let us consider a standard Brownian motion $W$ (with $W_0 = 0$\footnote{One can always consider the shifted Brownian motion $(B_t = W_{\tau_0 + t})_{t \ge 0}$ for a stopping time such that $W_{\tau_0} \sim \mu_0$ to go back to the context of Definition \ref{def:embedding}.})
		for the following two examples.
	
		\vspace{0.5em}
	
		\noindent $\mathrm{(i)}$ 
		Let $t_0 > 0$ and $\mu_s = \mathbf{N}(0, t_0 +s)$ be the Gaussian distribution with variance $t_0+s$ for all $s \in [0,1]$,
		then $(\mu_s)_{s \in [0,1]}$ is increasing in convex order.
		Moreover, for every $s \in [0,1]$, 
		the Root solution with one marginal $\mu_s$ is given by
		$$
			\sigma^1_s := \inf \big\{ t \ge 0 ~: (t, W_t) \in \Rc^1_s \big\}
			~\mbox{with}~
			\Rc^1_s := [t_0+s, \infty] \x \R.
		$$
		The barriers $(\Rc^1_s)_{s \in [0,1]}$ are ordered in the sense that
		$$
			\Rc^1_{s_1} \supseteq \Rc^1_{s_2},
			~~\mbox{whenever}~s_1 \le s_2.
		$$
		Consequently, one has $\sigma^1_{s_1} \le \sigma^1_{s_2}$ whenever $s_1 \le s_2$;
		and therefore,  $\sigma^{\infty}_s = \sigma^1_s$ for all $s \in [0,1]$,
		$$
			\tRoot^{\infty}_s 
			~=~ 
			\inf\left\{ t\geq 0:~ (t,W_t) \in \mathcal{R}^1_s \right\}
			~=~ 
			\inf\left\{ t\geq  \tRoot^{\infty}_{s-}:~ (t,W_t) \in \mathcal{R}^1_s \right\}.
		$$
		
		\noindent $\mathrm{(ii)}$ 
		Let $p: [0,1] \to [0, \frac12]$ be an increasing function,
		the family of marginal distributions $(\mu_s)_{s \in [0,1]}$ is defined by
		$$
			\mu_s (dx)~:=~ \big(1 - 2 p(s) \big) \delta_{0}  (dx) + p(s) \delta_{1} (dx) + p(s) \delta_{-1}(dx).
		$$
		Then it is clear that $(\mu_s)_{s \in [0,1]}$ is a family of marginals increasing in convex order.
		For each $s \in [0,1]$, the Root solution $\sigma^1_s$ to the one-marginal SEP with distribution $\mu_s$ is given by the hitting time
		$$
			\sigma^1_s ~=~ \inf \big\{ t \ge 0 ~: (t, W_t) \in \Rc^1_s \big\},
		$$
		where $\Rc^1_s$ takes the form, for some function $R_1: [0, 1] \to \R$,
		$$
			\Rc^1_s ~=~
			\big( [0, \infty] \x [1, \infty) \big) 
			~\cup~
			\big( [0, \infty] \x (- \infty, -1] \big) 
			~\cup~ 
			\big( [R_1(s), \infty] \x \{0\} \big).
		$$
		As $p(s)$ is an increasing function, it is easy to see that the function $R_1(s)$ is also increasing\footnote{In fact, one can compute explicitly the value of $R_1(s)$ for each $s \in [0,1]$,
		since the joint distribution of $ ( W_t, \min_{0 \le r \le t} W_r, \max_{0 \le r \le t} W_r)$ for each $t \ge 0$ can be described through an infinite series (see e.g. \cite{he1998double}).},
		so that one has
		 $\Rc^1_{s_1} \supset \Rc^1_{s_2}$ whenever $s_1 \le s_2$.
		In this case, one has, $\sigma^{\infty}_s = \sigma^1_s$ for all $s \in [0,1]$ so that
		$$
			\tRoot^{\infty}_s 
			~=~ 
			\inf\left\{ t\geq 0:~ (t,W_t) \in \mathcal{R}^1_s \right\}
			~=~ 
			\inf\left\{ t\geq  \tRoot^{\infty}_{s-}:~ (t,W_t) \in \mathcal{R}^1_s \right\}.
		$$
	\end{example}

	\begin{remark}
		It would be very interesting to formulate (sufficient) conditions on $(\mu_s)_{s \in [0,1]}$ so that the one-marginal Root barriers for $\mu_s$ are automatically ordered (as in \cite{AzemaYor} for Az\'ema-Yor embedding).
	\end{remark}	

	In the general case, when the convergence of set-valued process $\{ (\Rc^n_s)_{s \in [0,1]}, n \ge 1\}$ is unclear,
	one can still fix a constant $s \in [0,1]$, and consider the limit of $\{ \Rc^n_s, n \ge 1\}$.
	Indeed, for each fixed $s \in [0,1]$, by a tightness argument as in \cite{Root}, $\Rc^n_s$ converges to some limit barrier $\Rc^{\infty}_s$ along some subsequence.
	Combining the limit argument in Section \ref{subsec:cvg}, one has a countable set $N \subset [0,1)$ such that 
	$s \mapsto \sigma^{\infty}_s$ is continuous on $[0,1] \setminus N$,
	and $(\sigma^n_s, W_{\sigma^n_s}) \to (\sigma^{\infty}_s, W_{\sigma^{\infty}_s})$ weakly for every $s \in [0,1] \setminus N$, so that 
	$$
		(\sigma^{\infty}_s, W_{\sigma^{\infty}_s}) ~\in~ \Rc^{\infty}_s,
		~~\mbox{for every}~s \in [0,1] \setminus N.
	$$
	Nevertheless, this would not provide new information for \eqref{eq:faux_def_Root_full_marg}.
	In fact, notice that for a family $(\mu_s)_{s \in [0,1]}$ non-decreasing in convex order,
	the map $s \mapsto \mu_s$ is continuous except on at most countably many points,
	and in our context under Assumption \ref{assum:U}, $s \mapsto \mu_s$ is continuous on $[0,1]$.
	It follows that for a Root solution $((\sigma^{\infty}_s)_{s \in [0,1]}, W)$ with full marginals $(\mu_s)_{s \in [0,1]}$,
	one has
	$$
		\mathcal{L} \big ( W_{\sigma^{\infty}_{s-}} \big) 
		~=~
		\mathcal{L} \big( \lim_{r \nearrow s} W_{\sigma^{\infty}_r} \big)
		~=~
		\lim_{r \nearrow s} \mathcal{L} \big(W_{\sigma^{\infty}_r} \big)
		~=~
		\mu_s
		~=~
		\mathcal{L} \big ( W_{\sigma^{\infty}_{s}} \big).
	$$
	As $\sigma^{\infty}_{s-}$ is also a stopping time as supremum of a sequence of stopping times,
	$\sigma^{\infty}_{s-} \le \sigma^{\infty}_s$, and $(W_{\sigma^{\infty}_s \wedge t})_{t \ge 0}$ is uniformly integrable,
	it follows that $\sigma^{\infty}_{s-} = \sigma^{\infty}_s$, a.s. (\cite[Theorem 3]{Monroe}).
	This implies that, for every fixed $s \in [0,1]$, the characterization in \eqref{eq:faux_def_Root_full_marg} 
	is equivalent to
	\begin{equation} \label{eq:Root_charact}
		\tRoot^{\infty}_s ~=~ \inf\left\{ t\geq \tRoot^{\infty}_{s}:~ (t,W_t) \in \mathcal{R}_s \right\},
		~\mbox{a.s.}
	\end{equation}
	In particular, \eqref{eq:Root_charact} holds always true if we set $\mathcal{R}_s := [0, \infty] \x \R$,
	and in this sense  \eqref{eq:faux_def_Root_full_marg}  or  \eqref{eq:Root_charact}  would not really provide any information as characterization of the limit Root solution.

\section*{Acknowledgments}
The authors are grateful to two anonymous referees for helpful suggestions and comments.

\bibliographystyle{abbrvnat}

\end{document}